\documentclass[oneside]{amsart}

\usepackage{microtype, mathrsfs, xfrac, amssymb, xspace}
\usepackage{amsmath, amsthm, url}
\usepackage[latin1]{inputenc}
\usepackage{todonotes}
 \usepackage{tikz-cd}
\usepackage[all]{xy}
\newtheorem{theorem}{Theorem}[section]
\newtheorem{lemma}[theorem]{Lemma}
\newtheorem{proposition}[theorem]{Proposition}
\newtheorem{corollary}[theorem]{Corollary}

\usepackage{hyperref}
\hypersetup{
    colorlinks=true,
    linkcolor=green,
    filecolor=magenta,      
    urlcolor=blue,
}
\usepackage{todonotes}

\usepackage{pdfsync}

\theoremstyle{definition}
\newtheorem{notation}[theorem]{Notation}
\newtheorem{definition}[theorem]{Definition}
\newtheorem{example}[theorem]{Example}
\newtheorem{question}[theorem]{Question}
\newtheorem{construction}[theorem]{Construction}

\newtheorem{remark}[theorem]{Remark}

\newcommand{\hyph}{\text{-}}
\newcommand{\cat}[1]{\mathbf{#1}}

\DeclareMathOperator{\Spec}{Spec}
\DeclareMathOperator{\MaxSpec}{MaxSpec}

\DeclareMathOperator{\Tr}{Tr}

\DeclareMathOperator{\id}{id}
\DeclareMathOperator{\Pic}{Pic}

\newcommand{\et}{\textrm{\'et}}
\newcommand{\RR}{\mathbb{R}}
\newcommand{\FF}{\mathbb{F}}
\newcommand{\ZZ}{\mathbb{Z}}
\newcommand{\MM}{\mathbb{M}}
\renewcommand{\AA}{\mathbb{A}}
\newcommand{\CH}{\mathrm{CH}}
\newcommand{\PP}{\mathbb{P}}
\newcommand{\tensor}{\otimes}
\newcommand{\CC}{\mathbb{C}}
\newcommand{\RRP}{\RR P}
\newcommand{\Hoh}{\mathrm{H}}
\newcommand{\iso}{\cong}
\newcommand{\A}{\mathbb{A}}
\newcommand{\Aone}{{\A^1}}
\newcommand{\sh}[1]{\mathcal{#1}}
\newcommand{\sm}{\setminus}
\newcommand{\Gm}{\mathbb{G}_m}
\DeclareMathOperator{\Mor}{Mor}

\newcommand{\pr}{\operatorname{pr}}
\newcommand{\op}{\textrm{op}}

\newcommand{\isomto}{\overset{\cong}{\longrightarrow}}

\makeatletter
\newcommand*{\relrelbarsep}{.386ex}
\newcommand*{\relrelbar}{%
 \mathrel{%
 \mathpalette\@relrelbar\relrelbarsep
 }%
}
\newcommand*{\@relrelbar}[2]{%
 \raise#2\hbox to 0pt{$\m@th#1\relbar$\hss}%
 \lower#2\hbox{$\m@th#1\relbar$}%
}
\providecommand*{\rightrightarrowsfill@}{%
 \arrowfill@\relrelbar\relrelbar\rightrightarrows
}
\providecommand*{\leftleftarrowsfill@}{%
 \arrowfill@\leftleftarrows\relrelbar\relrelbar
}
\providecommand*{\xrightrightarrows}[2][]{%
 \ext@arrow 0359\rightrightarrowsfill@{#1}{#2}%
}
\providecommand*{\xleftleftarrows}[2][]{%
 \ext@arrow 3095\leftleftarrowsfill@{#1}{#2}%
}
\makeatother

\begin{document}

\title{Classifying spaces for \'etale algebras with generators}

\author{Abhishek Kumar Shukla}

\thanks{The first author was partially supported by a graduate fellowship from the Science and Engineering Research
 Board, India.}
\address{Department of Mathematics\\
 University of British Columbia\\
 Vancouver, BC V6T 1Z2\\Canada}
\email{abhisheks@math.ubc.ca}

\author{Ben Williams}

\address{Department of Mathematics\\
 University of British Columbia\\
 Vancouver, BC V6T 1Z2\\Canada}
\thanks{The second author was partially supported by an NSERC discovery grant}
\email{tbjw@math.ubc.ca}



%

\begin{abstract} 
 We construct a scheme $B(r; \AA^n)$ such that a map $X \to B(r; \AA^n)$ corresponds to a degree-$n$ \'etale algebra
 on $X$ equipped with $r$ generating global sections. We then show that when $n=2$, i.e., in the quadratic \'etale
 case, that the singular cohomology of $B(r; \AA^n)(\RR)$ can be used to reconstruct a famous example of
 S.~Chase and to extend its application to showing that there is a smooth affine $r-1$-dimensional $\RR$-variety on
 which there are \'etale algebras $\sh A_n$ of arbitrary degrees $n$ that cannot be generated by fewer than $r$
 elements. This shows that in the \'etale algebra case, a bound established by U.~First and Z.~Reichstein in \cite{first} is sharp.
\end{abstract}

\maketitle

\section{Introduction}
Given a topological group $G$, one may form the \textit{classifying space}, well-defined up to homotopy equivalence, as
the base space of any numerable principal $G$-bundle $EG \to BG$ where the total space is contractible,
\cite[Theorem 7.5]{Dold1963}. The space $BG$ is a universal space for $G$-bundles, in that the set
of homotopy classes of maps $[X, BG]$ is in natural bijection with the set of numerable principal $G$-bundles on $X$.

If $G$ is a finite nontrivial group, then $BG$ is necessarily infinite dimensional, \cite{Swan1960}, and so there is no
hope of producing $BG$ as a variety even over $\CC$. Nonetheless, as in \cite{Totaro}, one can approximate $BG$ by
taking a large representation $V$ of $G$ on which $G$ acts freely outside of a high-codimension closed set $Z$, and such
that $(V-Z)/G$ is defined as a quasiprojective scheme. The higher the codimension of $Z$ in $V$, the better an
approximation $(V-Z)/G$ is to the notional $BG$.

In this paper, we consider the case of $G=S_n$, the symmetric group on $n$ letters. The representations we consider as
our $V$s are the most obvious ones, $r$ copies of the permutation representation of $S_n$ on $\AA^n$. The closed loci we
consider are minimal: the loci where the action is not free. We use the language of \'etale algebras to give an
interpretation of the resulting spaces. Our main result, Theorem \ref{pr:mainProp}, says that the scheme
$B(r; \A^n):=(V-Z)/S_n$ produced by this machine represents ``\'etale algebras equipped with $r$ generating global
sections'' up to isomorphism of these data. The schemes $B(r; \A^n)$ are therefore in the same relation to the group
$S_n$ as the projective spaces $\PP^r$ are to the group scheme $\Gm$.

Section \ref{preliminaries} is concerned with preliminary results on generation of \'etale algebras. The main
construction of the paper, that of $B(r; \AA^n)$, is made in Section \ref{sec:MainConst}, and the functor it represents
is described. Since are working with schemes, and not in a homotopy category, the space $B(r; \AA^n)$ does not classify
bundles, rather it represents a functor of ``bundles along with chosen generators'', which we now explain.

A choice of $r$ global sections generating an \'etale algebra $\sh A$ of degree $n$ on a scheme $X$ corresponds to
a map $\phi: X \to B(r; \AA^n)$. While the map $\phi$ is dependent on the chosen generating sections, we show in Section
\ref{sect.homotopy} that if one is prepared to pass to a limit, in a sense made precise there, that the $\Aone$-homotopy
class of a composite $\tilde \phi: X \to B(r; \AA^n) \to B(\infty; \AA^n)$ depends only on the isomorphism class of
$\sh A$ and not the generators. As a practical matter, this means that for a wide range of cohomology theories, $E^*$,
the map $E^*(\tilde \phi)$ depends only on $\sh A$ and not on the generators used to define it.

In Section \ref{sec:MotCoh}, working over a field, we observe that the motivic cohomology, and therefore the Chow
groups, of the varieties $B(r; \AA^2)$ has already been calculated in \cite{dugger_hopf_2007}.

A degree-$2$ or \textit{quadratic} \'etale algebra $\sh A$ over a ring $R$ carries an involution $\sigma$ and a trace
map $\Tr: \sh A \to R$. There is a close connection between $\sh A$ and the rank-$1$ projective module $\sh L = \ker (\Tr)$. In Section
\ref{sec:Relation}, we show that the algebra $\sh A$ can be generated by $r$ elements if and only if the projective
module $\sh L$ can be generated by $r$ elements. 

A famous counterexample of S.~Chase, appearing in \cite{swan}, shows that there is a smooth affine $r-1$-dimensional
$\RR$-variety $\Spec R$ and a line bundle $\sh L$ on $\Spec R$ requiring $r$ global sections to generate. This shows a
that a bound of O.~Forster \cite{forster} on the minimal number of sections required to generate a line bundle on
$\Spec R$, namely $\dim R +1$, is sharp. In light of Section \ref{sec:Relation}, the same smooth affine $\RR$-variety of
dimension $r-1$ can be used to produce
\'etale algebras $\sh A$, of arbitrary degree $n$, requiring $r$ global sections to generate. This fact was observed
independently by M.~Ojanguren. It shows that a bound established by U.~First and Z.~Reichstein in \cite{first} is sharp
in the case of \'etale algebras: they can always be generated by $\dim R + 1$ global sections and this cannot be
improved in general. The details are worked out in Section \ref{sec:Chase}, and we incidentally show that the example of
S.~Chase follows easily from our construction of $B(r; \AA^2)$ and some elementary calculations in the singular
cohomology of $B(r; \AA^2)(\RR)$.

Finally, we offer some thoughts about determining whether the bound of First and Reichstein is sharp if one restricts to
varieties over algebraically closed fields.

\subsection{Notation and other preliminaries}
\begin{itemize}
\item All rings in this paper are assumed to be unital, associative, and commutative. 

\item $k$ denotes a base ring.

\item A \textit{variety} $X$ is a geometrically reduced, separated scheme of finite type over a field. We do not require the
base field to be algebraically closed, nor do we require varieties to be irreducible.
\item $C_2$ denotes the cyclic group of order $2$. 
\end{itemize}
 
We use the functor-of-points formalism (\cite[Part IV]{Eisenbudgeometryschemes2000}) heavily throughout, which is to say we view a scheme $X$ as the sheaf of sets it
represents on the big Zariski site of all schemes
\[ X(U) = \Mor_{\cat {Sch}}(U, X). \]



\section{\'Etale algebras}
\label{preliminaries}

Let $R$ be a ring and $S$ an $R$-algebra. Then there is a morphism of rings 
$\mu:S\otimes_R S^\op\to S$ sending $a\otimes b$ to $ab$. 
 We obtain an exact sequence 
\begin{align} \label{exact.seq}
0\to \ker(\mu) \to S\otimes_R S^\op \xrightarrow{\mu} S\to 0 
\end{align}

We recall (\cite[Chapter 4]{book.ford}) that an $R$-algebra $S$ is called \textit{separable} if $S$ is projective $S\otimes_R S^\op$-module.



\begin{definition}
 Let $R$ be a ring. A commutative $R$-algebra $S$ is called \textit{\'etale} if $S$ is a flat, separable,
 finitely presented $R$-algebra.
\end{definition}

\begin{proposition}
 Let $R$ be a commutative ring, and $S$ a commutative $R$-algebra. Then the following are equivalent:
 \begin{enumerate}
 \item \label{ai} $S$ is an \'etale $R$-algebra.
 \item \label{aii} $S$ is a finitely presented $R$-algebra and $\Spec S \to \Spec R$ is formally \'etale in the sense of
 \cite[Section 17.1]{EGA4.4}.
 \end{enumerate}
\end{proposition}
\begin{proof}
 By \cite[Corollary 4.7.3]{book.ford}, we see that \ref{ai} implies \ref{aii}. Conversely, a finitely presented and
 formally \'etale map is flat and unramified \cite[Corollaire 17.6.2]{EGA4.4}, and a finitely generated commutative unramified
 $R$-algebra $S$ is separable, \cite[Theorem 8.3.6]{book.ford}.
\end{proof}

\begin{definition}
 An $R$-algebra $S$ is called \textit{finite \'etale} if $S$ is an \'etale $R$-algebra and a finitely generated $R$-module. 
\end{definition}

\begin{remark}\label{rem.finite_present}
If $S$ is a finitely presented $R$-algebra that is finitely generated as an $R$-module then it is also finitely
presented as an $R$-module (\cite[1.4.7]{ega4-1}). Moreover, finitely presented and flat modules are projective
(\cite[tag 058Q]{deJongStacksProject2017}), so a finite \'etale algebra $S$ over $R$ is, in particular, a projective $R$-module
of finite rank.
\end{remark}

\begin{definition}
 We say that an \'etale algebra is of \textit{degree $n$} if the rank of $S$ as a projective $R$-module is $n$. A
 degree-$n$ \'etale algebra is necessarily finite \'etale.
\end{definition}

Over a ring $R$, and for any integer $n>0$, there exists the trivial rank-$n$ \'etale algebra $R^n$ with componentwise
addition and multiplication. The next lemma states that all \'etale algebras are \'etale-locally isomorphic to the
trivial one.
\begin{lemma}\label{lem.etalg}
Let $R$ be a ring and $S$ an $R$-algebra. The following statements are equivalent:
\begin{itemize}
 \item $S$ is an \'etale algebra of degree $n$. 
 \item There is a finite \'etale $R$-algebra $T$ such that $S\otimes_R T\cong T^n$ as $T$-algebras. 
\end{itemize}
\end{lemma}
A proof may be found in \cite[Corollary 1.1.16, Corollary 4.4.6, Proposition 4.6.11]{book.ford}.

We may extend this definition to schemes. Fix a ground ring $k$ throughout.

\begin{definition}
 Let $X$ be a $k$-scheme. Let $\mathcal{A}$ be a locally free sheaf of $\mathcal{O}_X$-algebras of constant rank $n$. We say that $\mathcal{A}$ is an \textit{\'etale $X$-algebra} or \'etale algebra
 over $X$ if for every open affine subset $U\subset X$ the $\sh O_X(U)$-algebra $\mathcal{A}(U)$ is an \'etale
 algebra. If the algebras $\sh A(U)$ are \'etale of rank $n$, we say $\sh A$ is a \textit{degree-$n$ \'etale algebra}.
\end{definition}


By Remark \ref{rem.finite_present} it is clear that a sheaf of degree-$n$ \'etale algebras $\mathcal{A}$ over $X$ is a
quasi-coherent sheaf of $\mathcal{O}_X$-modules.

If $X$ is a $k$-scheme and $n$ a positive integer, then there exists a trivial rank $n$ \'etale algebra
$\mathcal{O}^n_X$ with componentwise addition and multiplication.


\begin{lemma} Let $X$ be a $k$-scheme and $\mathcal{A}$ be a finitely presented, quasi-coherent sheaf of $\mathcal{O}_X$-algebras. Then the following are
equivalent:
	\begin{itemize}
			\item $\mathcal{A}$ is an \'etale $X$-algebra of degree $n$.
			\item There is an affine \'etale cover $\{ U_i\xrightarrow{f_i} X\}$ such that $f_i^*\mathcal{A}\cong \mathcal{O}_{U_i}^n$ as $\mathcal{O}_{U_i}$-algebras. 
		\end{itemize}
\end{lemma}
\begin{proof}
 This is immediate from Lemma \ref{lem.etalg}.
\end{proof}

\begin{definition}
 If $A$ is an algebra over a ring $R$, then a subset $\Lambda \subset A$ is said to \textit{generate} $A$ over
 $R$ if no strict $R$-subalgebra of $A$ contains $\Lambda$.
\end{definition}

If $\Lambda = \{ a_1, \dots, a_r\} \subset A$ is a finite subset, then the smallest subalgebra of $A$ containing
$\Lambda$ agrees with the image of the evaluation map $k[x_1, \dots, x_r] \overset{(a_1,\dots, a_r)}{\to} A$. Therefore,
saying that $\Lambda$ generates $A$ is equivalent to saying this map is surjective.

\begin{proposition} \label{pr:genIsLocal}
 Let $\Lambda = \{a_1, \dots, a_r\}$ be a finite set of elements of $A$, an algebra over a ring $R$. The following are
 equivalent:
 \begin{enumerate}
 \item $\Lambda$ generates $A$ as an $R$-algebra.
 \item There exists a set of elements $\{f_1, \dots, f_n\} \subset R$ that generate the unit ideal and such that, for
 each $i \in \{1, \dots, n\}$, the
 image of $\Lambda$ in $A_{f_i}$ generates $A_{f_i}$ as an $R_{f_i}$-algebra.
 \item For each $\mathfrak m \in \MaxSpec R$, the image of $\Lambda$ in $A_{\mathfrak m}$ generates $A_{\mathfrak m}$
 as an $R_{\mathfrak m}$-algebra.
 \item Let $k(\mathfrak m)$ denote the residue field of the local ring $R_{\mathfrak m}$. For each $\mathfrak m\in \MaxSpec R$, the image of $\Lambda$ in $A\otimes_R k(\mathfrak m)$ generates $A\otimes_R k(\mathfrak m)$
 as a $k({\mathfrak m})$-algebra.
 \end{enumerate}
\end{proposition}
\begin{proof}
 In the case of a finite subset, $\Lambda = \{ a_1, \dots, a_r\}$, the condition that $\Lambda$ generates $A$ is
 equivalent to the surjectivity of the evaluation map $R[x_1, \dots, x_r] \to A$.

 The question of generation is therefore a question of whether a certain map is an epimorphism in the category of
 $R$-modules, and conditions (2)-(4) are well-known equivalent conditions saying that this map is an epimorphism.
\end{proof}


Using Proposition \ref{pr:genIsLocal}, we extend the definition of ``generation of an algebra'' from the case where the base is affine to
the case of a general scheme.

\begin{definition} \label{def:genOfAlg} Let $\mathcal{A}$ be an algebra over a scheme $X$. For
 $ \Lambda \subset \Gamma(X,\mathcal{A})$ we say that $\Lambda$ \textit{generates} $\mathcal{A}$ if, for each open
 affine $U\subset X$ the $\mathcal{O}_X(U)$-algebra $\mathcal{A}(U)$ is generated by restriction of sections in
 $\Lambda$ to $U$.
\end{definition}

\subsection{Generation of trivial algebras}

Let $n \ge 2$ and $r \ge 1$. Consider the trivial \'etale algebra $\sh O_X^n$ over a scheme $X$. A global section of this algebra is equivalent to a
morphism $X \to \AA^n$, and an $r$-tuple $\Lambda$ of sections is a morphism $X \to \left(\AA^n\right)^r$. One might
hope that the subfunctor $\sh F \subseteq \left(\AA^n\right)^r$ of $r$-tuples of sections generating $\sh O_X^n$ as an
\'etale algebra is representable, and this turns out to be the case.

In order to define subschemes of $\left(\AA^n\right)^r$, it will be necessary to name coordinates: \[(x_{11}, x_{12},
\dots, x_{1n}, x_{21}, \dots, x_{2n}, \dots, x_{r1}, \dots, x_{rn}).\] It will also be useful to retain the grouping into
$n$-tuples, so we define $\vec x_l = (x_{l1}, x_{l2}, \dots, x_{ln})$.

\begin{notation}
 Fix $n$ and $r$ as above. For $(i,j) \in \{1, \dots, n\}^2$ with $i<j$, let $Z_{ij} \subset \left(\AA^n\right)^r$
 denote the closed subscheme given by the sum of the ideals $(x_{ki} - x_{kj})$ where $k$ varies from $1$ to $n$.

 Write $U(r; \AA^n)$, or $U(r)$ when $n$ is clear from the context, for the open subscheme of $\left( \AA^n \right)^r$
 given by 
 \[ U(r; \AA^n) = \left( \AA^n\right)^r - \bigcup_{i < j} Z_{ij} \]
\end{notation}

\begin{proposition}
 Let $n \ge 2$ and $r \ge 1$. The open subscheme $U(r; \AA^n ) \subset \left(\AA^n\right)^r$ represents the functor
 sending a scheme $X$ to $r$-tuples $(a_1, \dots, a_r)$ of global sections of $\sh O_X^n$ that generate it as an $\sh O_X$-algebra.
\end{proposition}
\begin{proof}
 Temporarily, let $\sh F$ denote the subfunctor of $\left(\AA^n\right)^r$ defined by
 \[ \sh F(X) = \{ \Lambda \subseteq \sh (\Gamma(X, O_X^n)^r \mid \Lambda \text{ generates } \sh O_X^n \}. \]
 It follows from Proposition \ref{pr:genIsLocal} and Definition \ref{def:genOfAlg} that $\sh F$ is actually a sheaf on
 the big Zariski site.

 Both $U(r; \AA^n)$ and $\sh F$ are subsheaves of the sheaf represented by $(\AA^n)^r$, and therefore in order to show
 they agree, it suffices to show $U(r; \AA^n)(R) = \sh F(R)$ when $R$ is a local ring.

Let $R$ be a local ring. The set $U(r; \AA^n)(R)$ consists of certain $r$-tuples $(\vec a_1, \dots, \vec a_r )$ of elements of
 $R^n$. Letting $a_{ki}$ denote the $i$-th element of $\vec a_k$, then the $r$-tuples are those with the property that for each
 $i \neq j$, there exists some $k$ such that $a_{ki} - a_{kj} \in R^\times$. The proposition now follows from Lemma \ref{lem:localTrivialGeneration} below.
\end{proof}

\begin{lemma} \label{lem:localTrivialGeneration}
 Let $R$ be a local ring, with maximal ideal $\mathfrak m$. Let $(\vec a_1, \dots, \vec a_r)$ denote an $r$-tuple of
 elements in $R^n$, and let $a_{ki}$ denote the $i$-th element of $\vec a_k$. The following are equivalent:
 \begin{itemize}
 \item The set $\{\vec a_1, \dots, \vec a_r \}$ generates the (trivial) \'etale $R$-algebra $R^n$.
 \item For each pair $(i,j)$ satisfying $1 \le i < j \le n$, there is some $k \in \{1 ,\dots, r\}$ such that the
 element $a_{ki} - a_{kj}$ is a unit in $R^\times$.
 \end{itemize}
\end{lemma}

\begin{proof}
 Each condition is equivalent to the same condition over $R/\mathfrak m$: the first by virtue of \ref{pr:genIsLocal},
 and the second by elementary algebra. Therefore, it suffices to prove this when $R$ is a field.

 Suppose $\{\vec a_1, \dots, \vec a_r \}$ generates $R^n$ as an algebra. Then, for any pair of indices $(i,j)$ with
 $1 \le i < j \le n$, it is possible to find a polynomial $p \in R[X_1, \dots, X_r]$ such that
 $p(a_{1i}, a_{2i}, \dots, a_{ri}) = 1$ and $p(a_{1j}, a_{2j}, \dots, a_{rj}) = 0$. In particular, there exists some
 $l$ such that $a_{li} \neq a_{lj}$.

 Conversely, suppose that for each pair $i<j$, we can find some $l$ such that $a_{li} \neq a_{lj}$.
 For each pair $i \neq j$, we can find a polynomial $p_{i,j} \in R[x_1, \dots, x_r]$ with the property that
 $p_{i,j}(a_{1i} , \dots, a_{ri}) = 1$ and $p_{i,j}(a_{1j}, \dots, a_{rj}) = 0$ by taking
 \[ p_{i,j}= (a_{ki}- a_{kj})^{-1}(x_k - a_{kj}) \]
 for instance. Consequently, we may produce a polynomial $p_i \in R[x_1, \dots, x_r]$ with the property that
 $p_i(a_{1j} , \dots, a_{rj}) = \delta_{i,j}$ (Kronecker delta). It follows that $\{\vec a_1, \dots, \vec a_r \}$
 generates the trivial algebra.
\end{proof}

\section{Classifying spaces} \label{sec:MainConst}

Fix $n \ge 2$ and $r \ge 1$. 


\begin{notation}
 For a given $k$-scheme $X$, a \textit{degree-$n$ \'etale algebra $\sh A$ with $r$ generating sections} denotes the
 data of a degree-$n$ \'etale algebra $\sh A$ over $X$, and an $r$-tuple of sections
 $(a_1, \dots, a_r) \in \Gamma( X, \sh A)$ that generate $\sh A$. These data will be briefly denoted
 $(\sh A , a_1, \dots, a_r)$. A \textit{morphism} $\psi: (\sh A, a_1, \dots, a_r) \to (\sh A', a_1' ,\dots, a_r')$ of
 such data consists of a map $\psi : \sh A \to \sh A'$ of \'etale algebras over $X$ such that $\psi(a_i) =a_i'$ for all
 $i \in \{1, \dots, r\}$. It is immediate that all morphisms are isomorphisms, and between any two objects, there is at
 most one isomorphism. The isomorphism class of $(\sh A, a_1, \dots, a_r)$ will be denoted $[\sh A, a_1, \dots, a_r]$.
\end{notation}

\begin{definition} \label{def:Fr}
 For a given $X$, there is a set, rather than a proper class, of isomorphism classes of degree-$n$ \'etale algebras
 over $X$, and so there is a set of isomorphism classes of degree-$n$ \'etale algebras with $r$ generating
 sections. Since generation is a local condition by Proposition \ref{pr:genIsLocal}, it follows that there is a functor
\begin{align*} \sh F(r;\AA^n) & : k\hyph\cat{Sch} \to \cat{Set}, \\ \sh F(r;\AA^n)(X) & = \{ [\sh A, a_1, \dots , a_r] \mid (\sh A, a_1,
\dots, a_r) \text{ is a degree-$n$} \\ & \text{ \'etale algebra over } X \text{ and $r$ generating sections} \} \end{align*}
\end{definition}

The purpose of this section is to produce a variety $B(r;\AA^n)$ representing the functor $\sh F(r;\AA^n)$ on the category of $k$-schemes.

\subsection{\texorpdfstring{Descent for $\sh F(r, \AA^n)$}{Descent for F(r;An)}}

\begin{proposition} \label{pr:FisSheaf}
 The functor $\sh F(r; \AA^n)$ is a sheaf on the big \'etale site of $\Spec k$.
\end{proposition}
In fact, it is a sheaf on the big fpqc site, but we will require only the \'etale descent condition.
\begin{proof}
 Suppose $X$ is a $k$-scheme and $\{ f_i : Y_i \to X \}_{i \in I}$ is an \'etale covering. We must identify $\sh F(r ;
 \AA^n)(X)$ with the equalizer in
 \[ E \to \prod_{i \in I} \sh F(r, \AA^n)(Y_i) \rightrightarrows \prod_{i,j \in I^2} \sh F(r; \AA^n)(Y_i \times_X
 Y_j). \]
 There is clearly a map $\sh F(r; \AA^n)(X) \to E$.

 Suppose we have an $i$-tuple of elements $([\sh A_i, \vec a(i) ])_{i \in I}$ in this equalizer. Choosing representatives in
 each case, we have degree-$n$ \'etale algebras $\sh A_i$ on each $Y_i$, along with chosen generating global
 sections. The equalizer condition is that there is an isomorphism over $Y_i \times_X Y_j$ of the form
 $\phi_{ij}: \pr_1^*(\sh A_i, \vec a(i)) \isomto \pr_2^*(\sh A_j, \vec{a}(j))$. The fact that there is at most one
 isomorphism between \'etale algebras with $r$ generating sections implies that we have a descent datum
 $( \sh A_i, \phi_{ij})$, and it is well known, \cite[Tag 023S]{deJongStacksProject2017}, that quasi-coherent sheaves satisfy
 \'etale descent. We therefore obtain a quasi-coherent sheaf of algebras $\sh A$ on $X$, and since $\sh A$ is an
 \'etale sheaf, the generating sections of each $\sh A_i$ glue to give generating sections of $\sh A$. This implies
 that the $\sh F(r; \AA^n)(X) \to E$ is surjective.

 To see it is injective, suppose $(\sh A, \vec{a})$ and $(\sh A', \vec{a'})$ become isomorphic when restricted to
 each $Y_i$. Then, since there can be at most a unique isomorphism between two \'etale algebras with generating
 sections, the local isomorphisms between $(\sh A, \vec{a})$ and $(\sh A', \vec{a'})$ assemble to give an
 isomorphism of descent data. Since there is an equivalence of categories between descent data and quasi-coherent
 sheaves, \cite[Tag 023S]{deJongStacksProject2017}, it follows that there is an isomorphism $\phi: \sh A \overset{\sim}{\to} \sh
 A'$. This isomorphism takes $\vec{a}$ to $\vec{a'}$, as required.
\end{proof}

\subsection{\texorpdfstring{Construction of $B(r; \AA^n)$}{Construction of B(r;An)}}


\begin{proposition} \label{pr:autisSn}
 Let $R$ be a nonzero connected ring. Then the
 automorphism group of the trivial \'etale $R$-algebra $R^n$ is the symmetric group $S_n$, acting on the terms.
\end{proposition}
\begin{proof}
 Since the equation $x^2-x=0$ has only the two solutions $1, 0$ in $R$, the condition $a^2=a$ for $a \in R^n$ implies that
 each component of $a$ is either $0$ or $1$. 

 Consider the elements
 \[ e_i = (0, \dots, 0, 1 , 0, \dots, 0) \in R^n. \]
 The set of these elements is determined by the conditions: $e_i^2 = e_i$, $e_i \neq 0$, $e_i e_j = 0$ for $i \neq j$
 and $\sum_{i=1}^n e_i = 1 \in R^n$.

 Therefore any automorphism of $R^n$ as an $R$-algebra permutes the $e_i$ and is determined by this permutation.
\end{proof}	

There is an action of the symmetric group $S_n$ on $\AA^n$, given by permuting the coordinates, and from there, there is
a diagonal action of $S_n$ on $\left( \AA^n \right)^r$, and the action restricts to the open subscheme $U(r; \AA^n)$.

\begin{proposition}
 The action of $S_n$ on $U(r; \AA^n)$ is scheme-theoretically free.
\end{proposition}
\begin{proof}
 It suffices to verify that the action is free on the sets $U(r;\AA^n)(K)$ where $K$ is a separably closed field over
 $k$. Here one is considering the diagonal $S_n$ action on $r$-tuples $(\vec a_1, \dots, \vec a_r)$ where each
 $\vec a_l \in K^n$ is a vector and such that for all indices $i \neq j$, there exists some $\vec a_l$ such that the
 $i$-th and $j$-th entries of $\vec a_l$ are different. The result follows.
\end{proof}

\begin{construction}
 There is a free diagonal action of $S_n$ on $U(r;\AA^n) \times \AA^n$, such that the projection
 $p: U(r;\AA^n) \times \AA^n \to U(r;\AA^n)$ is equivariant. The quotient schemes for these actions exist by reference
 to \cite[Expos\'e V, Proposition 1.8]{SGA1a} and \cite[Proposition 3.3.36]{Liu2002}. Write
 $q: E(r;\AA^n) \to B(r;\AA^n)$ for the induced map of quotient schemes. There is a commutative square
		\[
 \begin{tikzcd} U(r;\AA^n)\times \AA^n \arrow{d}{p} \arrow{r}{\pi'} & E(r;\AA^n) \arrow{d}{q} \\ U(r;\AA^n)
 \arrow{r}{\pi} & B(r;\AA^n)
		\end{tikzcd}
		\]
\end{construction}

\begin{proposition}
 In the notation above, the maps $\pi$ and $\pi'$ are finite.
\end{proposition}
\begin{proof}
 We concentrate on the case of $\pi$, that of $\pi'$ is similar. The map $\pi$ is formed as follows (see \cite[Expos\'e
 V, \S 1]{SGA1a}): it is possible to cover $U(r; \AA^n)$ by $S_n$-invariant open affine subschemes
 $\Spec R \subseteq \AA^{nr}$. Then $\pi|_{\Spec R} : \Spec R \to \Spec R^{S_n}$, induced by the inclusion of $R^{S_n}$
 in $R$. The map $R^{S_n} \to R$ is of finite type, since $R$ is of finite type over $k$. By \cite[Exercise 5.12,
 p68]{Atiyah1969}, the extension $R^{S_n} \to R$ is integral, and being of finite type, it is finite.
\end{proof}

\begin{corollary}
 The maps $\pi: U(r; \AA^n) \to B(r; \AA^n)$ and $\pi': U(r; \AA^n) \times \AA^n \to E(r; \AA^n)$ are $S_n$-torsors.
\end{corollary}

That is, each satisfies the conditions of \cite[Expos\'e V, Proposition 2.6]{SGA1a}.

\begin{remark}\label{rem.descent}
 The sheaf of sections of the map $p: U(r;\AA^n) \times \AA^n \to U(r;\AA^n)$ is the trivial degree-$n$ \'etale algebra $\sh
 O_{U(r;\AA^n)}^n$ on $U(r;\AA^n)$. The action of $S_n$ on these sections is by algebra automorphisms, and so the sheaf of sections
 of the quotient map $q: E(r;\AA^n) \to B(r;\AA^n)$ is endowed with the structure of a degree-$n$ \'etale algebra $\sh E(r;\AA^n)$ on
 $B(r;\AA^n)$. We will often confuse the scheme $E(r;\AA^n)$ over $B(r;\AA^n)$ with the \'etale algebra of sections $\sh E(r;\AA^n)$.
 
 The map $p$ has $r$ canonical sections $\{s_j \}_{j=1}^r$ given as follows:
 \[ s_j(\vec x_1, \vec x_2, \dots, \vec x_r)=((\vec x_1, \vec x_2, \dots, \vec x_r,), \vec x_j). \]

 These sections are $S_n$-equivariant, and so induce sections $\{t_i: B(r;\AA^n) \to E(r;\AA^n)\}_{i=1}^r$ of the map $q$.
\end{remark}
	
\begin{remark} By reference to \cite[Lemma 05B5]{deJongStacksProject2017}, the quotient $k$-scheme $B(r;\AA^n)$ is
 smooth over $k$ since $U(r;\AA^n)$ is and $\pi$ is faithfully flat
 (\cite[Expos\'e V, Proposition 2.6]{SGA1a}) and locally finitely presented. Since $\pi$ is finite it is a proper map. When the base $k$ is a field, the
 variety $B(r;\AA^n)$ is a quasiprojective variety but not projective. Indeed, if $B(r;\AA^n)\to \Spec(k)$ were proper
 then $U(r;\AA^n)\to \Spec(k)$ would be proper too, but $U(r;\AA^n)$ is a nonempty open subvariety of affine space.
\end{remark}

\subsection{The functor represented by $B(r,\AA^n)$}

We now establish the identity of functors $B(r,\AA^n)(X) = \sh F(r;\AA^n)(X)$.
\begin{construction} \label{con:presh}
 There is a canonical element $[\mathcal{E}(r;\AA^n),t_1,\ldots,t_r]$ in $\sh F(r;\AA^n)(B(r;\AA^n))$, see \ref{rem.descent}.
 Therefore, there exists a natural transformation of presheaves of $k$-schemes $B(r, \AA^n)(\cdot) \to \sh F(r; \AA^n)(\cdot)$ given by sending a
 map $\phi : X \to B(r; \AA^n)$ to the pullback of the canonical element.
\end{construction}

\begin{lemma}\label{fieldmap}
 If $[A,s_1,\ldots,s_r]\in \sh F(r;\AA^n)(R)$ where $R$ is a strictly henselian local ring, then there exists a unique
 morphism of schemes $\phi:\Spec(R)\to B(r;\AA^n)$ such that
 \[ [A,s_1,\ldots,s_r]=[ \phi^*(\mathcal{E}(r;\AA^n)),\phi^*t_1,\ldots,\phi^* t_r ] \]
\end{lemma}
\begin{proof}
 Since $R$ is a strictly henselian local ring, there exists an $R$-isomorphism $A\xrightarrow{\psi} R^n$ of algebras, by
 virtue of \cite[Proposition 1.4.4]{milneec}. Let
 $\{\psi(s_i) \} \subset R^n$ denote the corresponding sections of $R^n$.
 
 We thus obtain a map $\tilde{\phi}:\Spec(R)\to U(r;\AA^n)$ defined by giving the $R$-point
 $(\psi(s_1),\ldots,\psi(s_r))$. Post-composing this map with the projection $U(r;\AA^n) \to B(r;\AA^n)$, we obtain a morphism
 $\phi: \Spec(R)\to B(r;\AA^n)$. It is a tautology that $\phi^*(\mathcal{E}(r;\AA^n))=A$ and $\phi^*(t_i)=s_i$.

 It now behooves us to show that $\phi$ does not depend on the choices made in the construction.
 	
 Suppose $\phi': \Spec R \to B(r;\AA^n)$ is another morphism satisfying the conditions of the lemma. We may lift this
 $R$-point of $B(r;\AA^n)$ to an $R$-point $\tilde \phi' : \Spec R \to U(r;\AA^n)$, since $\pi$ is an \'etale
 covering, and therefore represents an epimorphism of \'etale sheaves \cite[Lemma 00WT]{deJongStacksProject2017}. By hypothesis we have
 \[ [A,s_1,\ldots,s_r]=[ \phi'^*(\mathcal{E}(r;\AA^n)),\phi'^*t_1,\ldots,\phi'^* t_r ]. \]
 Thus $\tilde \phi$ and $\tilde \phi'$ differ by an automorphism of $R^n$, i.e., by an element of $S_n$ since local
 rings are connected so \ref{pr:autisSn} applies. Therefore $\phi=\phi'$ as
 required.
\end{proof}

\begin{theorem} \label{pr:mainProp}
 If $X$ is a $k$-scheme, then the map \[ B(r, \AA^n)(X) \to \sh F(r; \AA^n)(X)\] is a bijection.
\end{theorem}
\begin{proof}
 We note that $B(r, \AA^n)$ represents a sheaf on the big \'etale site of $\Spec k$, since it is a $k$-scheme. The
 presheaf $\sh F(r; \AA^n)$ is also an \'etale sheaf, by virtue of Proposition \ref{pr:FisSheaf}. It therefore suffices to
 prove that $B(r, \AA^n)(\Spec R) \to \sh F(r; \AA^n)(\Spec R)$ when $R$ is a strictly henselian local ring, but this
 is Lemma \ref{fieldmap}.
\end{proof}

		
		
 

\begin{example}
 Let us consider the toy example where $k$ is a field and $X = \Spec K$ for some field extension $K/k$, where $n \ge 2$, and where $r=1$. That is, we are considering
 \'etale algebras $A/K$ along with a chosen generating element $a \in A$. After base change to the separable closure,
 $K^s$, we obtain a $S_n$-equivariant isomorphism of $K^s$-algebras:
 \[\psi : A_{K^s} \isomto (K^s)^{\times n}.\] For the sake of the exposition, use $\psi$ to identify source and
 target. The element $a \in A$ yields a chosen generating element $\tilde a \in (K^s)^n$. The element $\tilde a$
 is a vector of $n$ pairwise distinct elements of $K^s$. The element $\tilde a$ is a $K^s$-point of $U(1; \AA^n)$. In
 general, this point is not defined over $K$, but its image in $B(1;\AA^n)$ is. 

 Since $U(1;\AA^n) \subseteq \A^n$, and $B(1;\AA^n)=U(1;\AA^n)/\Sigma_n$, the image of $\tilde a$ in $B(1;\AA^n)(K^s)$ may be presented as the
 elementary symmetric polynomials in the $a_i$. To say that the image of $\tilde a = (a_1, \dots, a_n)$ in $B(1;\AA^n)$ is
 defined over $K$ is to say that the coefficients of the polynomial $\prod_{i=1}^n (x-a_i)$ are defined in $K$.

 The variety $B(1;\AA^n)$ is the $k$-variety parametrizing degree-$n$ polynomials with distinct roots, i.e., with
 invertible discriminant.
\end{example}

\begin{example}
 To reduce the toy example even further, let us consider the case of $k=K$ a field of characteristic different from $2$, and $n=2$.

 The variety $B(1;\AA^2)$ may be presented as spectrum of the $C_2$-fixed subring of $k[x,y, (x-y)^{-1}]$ under the action
 interchanging $x$ and $y$. This is $k[(x+y), (x-y)^2, (x-y)^{-2}]$, although it is more elegant to present it after
 the change of coordinates $c_1= x+ y$ and $c_0=xy$:
 \[ B(1;\AA^2) = \Spec k [c_1, c_0, (c_1^2 - 4c_0)^{-1} ] \]

 A quadratic \'etale $k$-algebra equipped with the generating element $a$ corresponds to the point $(c_1, c_0) \in
 B(1;\AA^2)(k)$ where $a$ satisfies the minimal polynomial $a^2 - c_1 a + c_0 = 0$.

 For instance if $k=\RR$, the quadratic \'etale algebra of complex numbers $\CC$ with generator $s+ti$ over $\RR$ (here $t \neq 0$), corresponds to the point
 $(2s, s^2+t^2) \in B(1;\AA^2)(\RR)$, whereas $\RR \times \RR$, generated by $(s+t,s-t)$ over $\RR$ (again $t\neq 0$) , corresponds to the point $(2s,s^2-t^2)$. 
\end{example}


\section{Stabilization in cohomology}
\label{sect.homotopy}

We might wish to use the schemes $B(r;\AA^n)$ to define cohomological invariants of \'etale algebras. The idea is the
following: suppose given such an algebra $\sh A$ on a $k$-scheme $X$, and suppose one can find generators $(a_1, \dots
a_r)$ for $\sh A$. Then one has a classifying map $\phi: X \to B(r;\AA^n)$, and one may apply a cohomology functor $E^*$, such as
Chow groups or algebraic $K$-theory, to obtain ``characteristic classes'' for $\sh A$-along-with-$(a_1, \dots, a_r)$, in
the form of $\phi^* : E^*(B(r;\AA^n)) \to E^*(X)$. The dependence on the specific generators chosen is a nuisance, and we see
in this section that this dependence goes away provided we are prepared to pass to a limit ``$B(\infty)$'' and assume
that the theory $E^*$ is $\Aone$-invariant, in that $E^*(X) \to E^*(X \times \Aone)$ is an isomorphism.

\begin{definition} \label{def:stabilizeMaps}
 There are \textit{stabilization} maps $U(r;\AA^n)\to U(r+1;\AA^n)$ obtained by augmenting an $r$-tuple of $n$-tuples
 by the $n$-tuple $(0,0,\dots,0)$. These stabilization maps are $S_n$-equivariant and therefore descend to maps
 $B(r;\AA^n) \to B(r+1;\AA^n)$. 

 The stabilization maps defined above may be composed with one another, to yield maps $B(r;\AA^n) \to B(r';\AA^n)$ for all $r <
r'$. These maps will also be called \textit{stabilization} maps. 
\end{definition}
	
\begin{proposition}\label{prop.htylem}
 Let $X$ be a $k$-scheme. Suppose 
 \begin{equation*}
 [\mathcal{A},a_1,\ldots,a_r] \in \sh F(r;\AA^n)(X) \quad \text{and} \quad
[\mathcal{A}',a_1',\ldots,a_{r'}'] \in \mathcal{F}(r'; \AA^n)(X)
 \end{equation*}have the property that $\sh A \iso \sh A'$ as \'etale
 algebras. Let $\phi: X\to B(r;\AA^n)$ and $\phi':X\to B(r';\AA^n)$ be the corresponding classifying morphism. For $R= r
 + r'$, the composite 
maps $\tilde \phi: X \to B(r;\AA^n) \to B(R;\AA^n)$ and $\tilde \phi': X \to B(r';\AA^n) \to B(R;\AA^n)$ given
 by stabilization are na\"ively $\Aone$-homotopic.
\end{proposition}

An ``elementary $\Aone$-homotopy'' between maps $\phi, \phi': X \to B$ is a map $\Phi: X \times \Aone \to B$
specializing to $\phi$ at $0$ and $\phi'$ at $1$. Two maps $\phi, \phi': X \to B$ are ``naively $\Aone$-homotopic'' if
they may be joined by a finite sequence of elementary homotopies. Two naively homotopic maps between smooth finite-type
$k$-schemes are identified in the $\mathbb{A}^1$-homotopy theory of schemes of \cite{morel}, but they do not account for all
identifications in that theory.

\begin{proof} We may assume that $\sh A = \sh A'$. We may also assume that $r=r'$---if $r < r'$, then pad the vector
 $(a_1, \dots, a_r)$ with $0$s to produce a vector $(a_1, \dots, a_r, 0, \dots , 0)$ of length $r'$, and similarly in
 the other case.

 Write $t$ for the parameter of $\Aone$. Let $\sh A[t]$ denote the pull-back of $\sh A$ along the projection $X \times
 \Aone \to X$.

 Consider the sections $((1-t)a_1, \dots, (1-t)a_r, ta'_1, \dots, ta'_r)$ of $\sh A[t]$. Since either $t$ or $(1-t)$ is
 a unit at all local rings of points $\AA^1$, by appeal to Proposition \ref{pr:genIsLocal} and consideration of the
 restrictions to $X \times (\Aone -\{0\})$ and $X \times (\Aone - \{1\})$, we see that $((1-t)a_1,
 \dots, (1-t)a_r, ta'_1, \dots, ta'_r)$ furnish a set of generators for $\sh A[t]$. At $t=0$, they specialize to $(a_1,
 \dots , a_r, 0, \dots, 0)$, viz., the generators specified by the stabilized map $\phi: X \to B(r;\AA^n) \to B(2r;\AA^n)$. At
 $t=1$, they specialize to $(0, \dots, 0, a'_1, \dots, a'_r)$, which is not precisely the list of generators specified
 by $\phi': X \to B(r;\AA^n) \to B(2r;\AA^n)$, but may be brought to this form by another elementary $\Aone$-homotopy.
\end{proof}

\begin{corollary}\label{cor.htylem} Let $\phi$ and $\phi'$ be as in the previous proposition. If $E^*$ denotes any $\Aone$-invariant
 cohomology theory, then $E^*(\tilde \phi)=E^*(\tilde \phi')$.
\end{corollary}


\section{The motivic cohomology of the spaces $B(r; \A^2)$} \label{sec:MotCoh}

For this section, let $k$ denote a fixed field of characteristic different from $2$. The motivic cohomology of the
spaces $B(r; \AA^2)$ has already been calculated in \cite{dugger_hopf_2007}.

\subsection{Change of coordinates}

\begin{lemma}\label{lem.weakequiv}
 There is an equivariant isomorphism $ U(r;\AA^2) \cong \mathbb{A}^r\setminus \{0\} \times \mathbb{A}^r$, where $C_2$ acts as
 multiplication by $-1$ on first factor $\mathbb{A}^r\setminus \{0\}$ and trivially on the second factor
 $\mathbb{A}^r$. Taking quotient by $C_2$-action yields
 $B(r;\AA^2) \cong (\mathbb{A}^r\setminus \{0\})/C_2\times \mathbb{A}^r.$
\end{lemma}

\begin{proof}
 By means of the change of coordinates
 \[ x_i - y_i = z_i , \quad x_i + y_i = w_i \]
 we see that $U(r; \AA^2) \iso (\A^r \sm \{0\}) \times \A^r$. Moreover, the action of $C_2$ on $U(r;\AA^2)$ is given by
 $z_i \mapsto -z_i$ and $w_i \mapsto w_i$. We therefore obtain an isomorphism
 $B(r; \AA^2)= U(r;\AA^2)/ C_2 \iso (\A^r \sm \{0\}) / C_2 \times \A^r$. Write $V(r;\AA^2)$ for $\A^r \sm \{0\}/C_2$. It is immediate that
 $B(r;\AA^2) \iso V(r;\AA^2) \times \AA^r$, and so there is a split inclusion $V(r;\AA^2) \to B(r;\AA^2)$ which is moreover an
 $\Aone$-equivalence.
\end{proof}

\subsection{The deleted quadric presentation}

\begin{definition}
 Endow $\PP^{2r-1}$ with the projective coordinates $a_1, \dots, a_r$, $b_1 , \dots, b_r$. Let $Q_{2r-2}$ denote the
 closed subvariety given by the vanishing of $\sum_{i=1}^r a_i b_i$, and let $DQ_{2r-1}$ denote the open complement
 $\PP^{2r-1} \setminus Q_{2r-2}$.
\end{definition}

The main computation of \cite{dugger_hopf_2007} is a calculation of the modulo-$2$ motivic cohomology of $DQ_{2r-1}$,
and of a family of related spaces $DQ_{2r}$. Our reference for the motivic cohomology of $k$-varieties
is \cite{MazzaLecturenotesmotivic2006}. For a given abelian group $A$, either $\ZZ$ or $\FF_2$ in this paper, and a
given variety $X$, the motivic cohomology
$\Hoh^{*,*}(X; A)$ is a bigraded algebra over the cohomology of the ground field, $\Spec k$.

Denote the modulo-$2$ motivic cohomology of $\Spec k$ by $\MM_2$. This is a
bigraded ring, 
\[\MM_2 = \bigoplus_{i,n} \MM_2^{n,i},\] nonzero only in degrees $0 \le n \le i$. There are two notable classes,
$\rho \in \MM^{1,1}_2$, the reduction modulo $2$ of $\{-1\} \in K^M_1(k) = \Hoh^{1,1}(\Spec k, \ZZ)$, and
$\tau \in \MM^{0,1}_2$, corresponding to the identity $(-1)^2 = 1$. If $-1$ is a square in $k$, then $\rho = 0$, but
$\tau$ is always a nonzero class.

\begin{proposition}[Dugger--Isaksen, \cite{dugger_hopf_2007} Theorem 4.9] \label{pr:DIone}
 There is an isomorphism of graded rings
 \[ \Hoh^{\ast, \ast} (DQ_{2r-1}; \FF_2) \iso \frac{\MM_2 [a, b]}{ (a^2 - \rho a - \tau b, b^{r})} \]
 where $|a| = (1,1)$ and $|b| = (2,1)$.

 Moreover, the inclusion $DQ_{2r-1} \to DQ_{2r+1}$ given by $a_{r+1} = b_{r+1} = 0$ induces the map $\Hoh^{\ast, \ast}
 (DQ_{2r+1}; \FF_2) \to \Hoh^{\ast, \ast} (DQ_{2r-1}; \FF_2)$ sending $a$ to $a$ and $b$ to $b$.
\end{proposition}

This proposition subsumes two other notable calculations of invariants. In the first place, owing to the
Beilinson--Lichtenbaum conjecture \cite{voe}, it subsumes the calculation of $\Hoh^\ast_{\et} (DQ_{2r-1};\FF_2)$. For instance,
if $k$ is algebraically closed, then $\MM_2 = \FF[\tau]$, and one deduces that $\Hoh^\ast_{\et}( DQ_{2r-1}; \FF_2) \iso
\FF_2[a, b] /( a^2 -b , b^{r}) = \FF_2[a]/(a^{2r})$.

In the second, since $\Hoh^{2n,n}(\cdot; \FF_2)$ is identified with $\CH^n(\cdot) \tensor_\ZZ \FF_2$, the calculation of
the proposition subsumes that of the Chow groups modulo $2$. In fact, the extension problems that prevented Dugger and
Isaksen from calculating $\Hoh^{\ast, \ast}(DQ_{2r-1} ; \ZZ)$ do not arise in this range, and by reference to the
appendix of \cite{dugger_hopf_2007}, which in turn refers to \cite{KarpenkoChowgroupsprojective1990}, one can calculate
the integral Chow rings. This is done in the first two paragraphs of the proof of \cite[Theorem 4.9]{dugger_hopf_2007}.

\begin{proposition} \label{pr:DItwo}
 One may present
 \[ \CH^\ast(DQ_{2r-1}) = \frac{\ZZ[\tilde b]}{(2\tilde b, \tilde b^{r})}, \quad |b| = 1.\]
 As before, the map $DQ_{2r-1} \to DQ_{2r+1}$ given by adding $0$s induces the map $b \mapsto b$ on Chow
 rings. Moreover 
 $\CH^\ast(DQ_{2r-1}) \tensor_\ZZ \FF_2$ can
 be identified with the subring of $ \Hoh^{\ast, \ast} (DQ_{2r-1}; \FF_2)$ generated by $b$.
\end{proposition}

The reason we have explained all this is that there is a composite of maps
\begin{equation}
 \label{eq:1}
 DQ_{2r-1} \to (\AA^{r} \sm \{0\})/ C_2 \to B(r; \A^2),
\end{equation}
both of which are $\Aone$-equivalences, and so Propositions \ref{pr:DIone} and \ref{pr:DItwo} amount to a
calculation of the motivic and \'etale cohomologies and Chow rings of $B(r; \AA^2)$. Both maps in diagram \eqref{eq:1} are
compatible in the evident way with an increase in $r$, so that we may use the material of this section to compute the
stable invariants of $B(r; \AA^2)$ in the sense of Section \ref{sect.homotopy}.

The $\Aone$-equivalence $B(r; \AA^2) \to (\AA^{r} \sm \{0\})/ C_2$ was constructed above in Lemma \ref{lem.weakequiv}, so it
remains to prove the following.

\begin{lemma} \label{lem:AffDQ}
 Let $r \ge 1$. The variety $DQ_{2r-1}$ is affine and has coordinate ring
 \begin{equation}
 \label{eq:3}
 R = \left[\frac{k[x_1,\dots, x_r, y_1, \dots, y_r]}{\big(1-\sum_{i=1}^r x_iy_i
 \big)}\right]^{C_2} 
 \end{equation}
 where the $C_2$ action on $x_i$ and $y_i$ is by $x_i \mapsto -x_i$ and $y_i \mapsto -y_i$.
\end{lemma}
\begin{proof}
 The variety $DQ_{2r-1}$ is a complement of a hypersurface in $\PP^{2r-1}$, and is therefore affine.

 Let $Q$ denote
 $a_1 b_1 + \dots + a_r b_r$. The coordinate ring of $DQ_{2r-1}$ is the ring of degree-$0$ terms in the graded ring
 $S=k[a_1, \dots, a_r, b_1, \dots, b_r, Q^{-1}]$, where $|a_i| = |b_i| = 1$ and $|Q^{-1}| = -2$. This ring is the subring
 of $S$ generated by the terms $a_ia_j Q^{-1}$, $a_i b_j Q^{-1}$ and $b_i b_j Q^{-1}$.

 Consider the ring 
\begin{equation} \label{eq:TasPresnted} T = \frac{k[ x_1, \dots, x_r, y_1 , \dots , y_r]}{( 1- \sum_{i=1}^r x_i y_i )}.
\end{equation}

One may define a map of
 rings $\phi: S \to T$ by sending $a_i \mapsto x_i$ and $b_i \mapsto y_i$, since $Q \mapsto 1$ under this
 assignment. Restricting to $\Gamma(DQ_{2r-1}, \sh O_{DQ_{2r-1}}) \subset S$, one obtains a map $\Gamma(DQ_{2r-1}, \sh O_{DQ_{2r-1}}) \to T$ for which the image is precisely the subring
 generated by terms $x_ix_j$, $x_i y_j$ and $y_i y_j$, i.e., the fixed subring under the $C_2$ action given by $x_i
 \mapsto -x_i$ and $y_i \mapsto -y_i$.

 It remains to establish this map is injective. We show that the kernel of the map $\phi: S \to T$ contains only one
 homogeneous element, $0$, so that the restriction of this map to the subring of degree-$0$ terms in $S$ is
 injective. The kernel of $\phi$ is the ideal $(Q-1)$. Since $S$ is an integral domain, degree considerations imply that
 no nonzero multiple of $(Q-1)$ is homogeneous.
\end{proof}

\begin{proposition} \label{pr:Jouan}
 For all $r$, there is an $\Aone$-equivalence \[DQ_{2r-1} \to (\AA^{r} \sm \{0\})/ C_2.\]
\end{proposition}

\begin{proof}
 Let $T$ be as in the proof of Lemma \ref{lem:AffDQ}. It is well known that $\Spec T$ is an affine vector bundle torsor
 over $\AA^r \setminus \{0\}$. In fact, for each $j \in \{1, \dots, r\}$, if we define
 $U_j \iso \AA^1\sm\{0\} \times \AA^{r-1}$ to be the open subscheme of $\AA^r \sm \{0\}$ where the $j$-th coordinate is
 invertible, then we arrive at a pull-back diagram
 \[ \xymatrix{ \AA^{r-1} \times U_j \iso \Spec T \times_{\AA^r \setminus \{0\}} U_j \ar[r] \ar[d] & \Spec T \ar[d] \\ U_j
 \ar[r] & \AA^r \setminus \{0\} } \]
 Since $U_j$ inherits a free $C_2$-action, it follows that in the quotient we obtain a vector bundle $(\AA^{r-1} \times
 U_j)/C_2 \to U_j/C_2$, and so the map $(\Spec T) / C_2 \to (\AA^r \sm \{0\})/ C_2$ is an $\AA^1$-equivalence, as claimed.
\end{proof}

As a consequence of Proposition \ref{pr:Jouan} we observe that the affine variety $DQ_{2r-1}$ is an affine approximation of $B(r;\AA^2)$. 


\section{Relation to line bundles in the quadratic case} \label{sec:Relation}

We continue to work over a field $k$, and to require that the characteristic of $k$ be different from $2$.

In the case where $n=2$, the structure group of the degree-$n$ \'etale algebra is $C_2$, the cyclic group of order $2$,
which happens to be a subgroup of $\Gm$. More explicitly, $\Hoh^1_{\et}(\Spec R ; C_2)$ is an abelian group which is isomorphic to the isomorphism classes of quadratic \'etale algebras on $\Spec R$. On the other hand due to the Kummer sequence and $C_2\subset \Gm$ we have 
\[ 0\to R^*/R^{*2} \to \Hoh^1_{\et}(\Spec R ; C_2)\to {}_2\Pic(R)\to 0 \]
which means that $\Hoh^1_{\et}(\Spec R ; C_2)$ is identified with the set of isomorphism classes of $2$-torsion line bundles $\sh L$ with a choice of trivialization $\phi:\sh L \otimes \sh L \isomto \sh O_R$. 

 This is the basis of the following construction.

\begin{construction} \label{cons:LfromA} Let $X$ be a scheme such that $2$ is invertible in all residue fields, and let
 $\sh A$ be a quadratic \'etale algebra on $X$. There is a trace map \cite[Section I.1]{knus}:
 \[ \Tr : \sh A \to \sh O \]
 and an involution $\sigma: \sh A \to \sh A$ given by $\sigma = \Tr - \id$. Define $\sh L$ to be the kernel of $\Tr:
 \sh A \to \sh O$. The sequence of sheaves on $X$
 \begin{equation}
 \label{eq:4}
 0 \to \sh L \to \sh A \to \sh O \to 0
 \end{equation}
 is split short exact, where the splitting $\sh O \to \sh A$ is given on sections by $x \mapsto \frac{1}{2} x$.

 The construction of $\sh L$ from $\sh A$ gives an explicit instantiation of the map $\Hoh^1_{\et}(X , C_2) \to
 \Hoh^1_{\et} (X, \Gm)$ on isomorphism classes. We note that $\sh L$ must necessarily be a $2$-torsion line
 bundle, in that $\sh L \tensor \sh L $ is trivial.
\end{construction}

It is partly possible to reverse the construction of $\sh L$ from $\sh A$.
\begin{construction} \label{cons:AfromL}
 Let $X$ be as above, and let $\sh L$ be a line-bundle on $X$ such that there is an isomorphism $\sh L \tensor \sh L
 \to \sh O$. Let $\phi: \sh L \tensor \sh L \to \sh O$ be a specific choice of
 isomorphism. From the data $(\sh L, \phi)$, we may produce an \'etale algebra $\sh A = \sh O \oplus \sh L$ on which
 the multiplication is given, on sections, by $(r, x) \cdot (r', x') = (rr' + \phi(x \tensor x'), rx' + r' x)$. 
\end{construction}

\begin{proposition} \label{pr:generationEquivalence}
 Let $X$ be a scheme such that $2$ is invertible in all residue fields of points of $X$. Let $\sh A$ a quadratic
 \'etale algebra on $X$. Let $\sh L$ be the associated line bundle to $\sh A$, as in Construction
 \ref{cons:LfromA}. Suppose $a_1, \dots, a_r$ are global sections of $\sh A$. Then $a_1, \dots, a_r$ generate
 $\sh A$ as an algebra if and only if $a_1 - \frac{1}{2} \Tr(a_1), \dots, a_r - \frac{1}{2} \Tr(a_r)$ generate $\sh L$
 as a line bundle.
\end{proposition}
\begin{proof}
 Write $q$ for the map $a \mapsto a - \frac{1}{2} \Tr(a)$. The questions of generation of $\sh A$ and of $\sh L$ may be
 reduced to residue fields at points of $X$, by Proposition \ref{pr:genIsLocal} for the algebra and Nakayama's lemma
 for the line bundle.

 We may therefore suppose $F$ is a field of characteristic different from $2$, and that $A/ F$ is a quadratic \'etale algebra.
 Since $2$ is invertible, we may write $A = F[z]/(z^2 - c)$ for some element $c \in F^\times$. In this
 presentation, $\sigma(z) = -z$ and $\Tr(az + b) = 2b$. The kernel of the trace map, i.e. $\sh L$, is therefore $Fz$. The map $q: A
 \to Fz$ is given by $q(a z + b) = az$.

 An $r$-tuple $\vec a = (a_1 z + b_1, \dots , a_r z + b_r)$ of elements of $A$ generate it as an $F$-algebra if and only
 if $q(\vec a) = (a_1z , \dots a_rz)$ do. This tuple generates $A$ as an algebra if and only if at least one of the
 $a_i$ is nonzero, which is exactly the condition for it to generate $Fz$ as an $F$-vector space
 %
\end{proof}

\begin{remark}
 Let $k$ be a field of characteristic different from $2$. Let $X$ be a $k$-variety. An \'etale algebra of degree $2$ generated by $r$ global
 sections corresponds to a map $X \to B(r; \AA^2)$. A line bundle generated by $r$ global sections corresponds to a map
 $X\to \PP^{r-1}$. In the light of Proposition \ref{pr:generationEquivalence}, there must be a map of varieties $B(r;
 \AA^2) \to \PP^{r-1}$. This map is given by 
 \[ B(r; \AA^2) \overset{\iso}{\to } (\AA^r\sm\{0\})/C_2 \times \AA^r \overset{p_1}{\to} (\AA^r \sm \{0\}) / C_2\to
 (\AA^r \sm \{0\} )/ \Gm \isomto \PP^{r-1} \]
 where the morphisms are, left to right, the isomorphism of Lemma \ref{lem.weakequiv}, projection onto the second
 factor, and the map induced by the inclusion $C_2 \subset \Gm$.
\end{remark}


\section{The example of Chase} \label{sec:Chase}

The following will be referred to as ``the example of Chase''.

\begin{construction} \label{cons:Chase}
 Let $S = \RR[ z_1, \dots, z_r ]/\Big( \sum_{i=1}^r z_i^2 - 1 \Big)$
 and equip this with the $C_2$-action given by $z_i \mapsto -z_i$. Let $R = S^{C_2}$. The dimension of both $R$ and $S$ is $r-1$.

 The ring $R$ carries a projective module of rank $1$, i.e., a line bundle, that requires $r$ global sections in order
 to generate it. This example given in \cite[Theorem 4]{swan}.
\end{construction}

\begin{remark}
 In fact, the line bundle in question is of order $2$ in the Picard group, so Proposition
 \ref{pr:generationEquivalence} applies and there is an associated quadratic \'etale algebra on $\Spec R = Y(r)$
 requiring $r$ generators. The algebra is, of course, dependent on a choice of trivialization of the square of the
 line bundle, but one may choose the trivialization so the \'etale algebra in question is $S$ itself as an $R$-algebra.
\end{remark}

\begin{remark}
 This construction shows that the bound of First and Reichstein, \cite{first}, on the number of generators required by
 an \'etale algebra of degree $2$ is tight. This was first observed, to the best of our knowledge, by M.~Ojanguren in private communication.

 Even better, replacing $S$ by $S \times R^{n-2}$ over $R$, one produces a degree-$n$ \'etale algebra over $R$
 requiring $r$ elements to generate, so the bound is tight in the case of \'etale algebras of arbitrary degrees. We owe
 this observation to Zinovy Reichstein.
\end{remark}

The original method of proof that the line bundle in the example of Chase cannot be generated by fewer than $r$ global
sections uses the Borsuk--Ulam theorem. Here we show that a variation on that proof follows naturally from our general
theory of classifying objects. The Borsuk--Ulam theorem is a theorem about the topology of $\RRP^r$, so it can be no
surprise that it is replaced here by facts about the singular cohomology of $\RRP^r$.

\subsection{The homotopy type of the real points of \texorpdfstring{$B(r;\AA^2)$}{B(r,A2)}}

In addition to the general results about the motivic cohomology of $B(r; \AA^2)$, we can give a complete description of
the homotopy type of the real points $B(r; \AA^2)(\RR)$.

If $X$ is a nonsingular $\RR$-variety, then it is possible to produce a complex
manifold from $X$ by first extending scalars to $\CC$ and then employing the usual Betti realization functor to produce
a manifold $X(\CC)$. Since $X$ is defined over $\RR$, however, the resulting manifold is equipped with an action of the
Galois group $\operatorname{Gal}(\CC/\RR) \iso C_2$. We write $X(\RR)$ for the Galois-fixed points of $X(\CC)$.

\begin{remark}\label{rem.finprod}
 The real realization functor $X \leadsto X(\RR)$ preserves finite products, so that if $f,g: X\to Y$ are two maps of
 varieties and $H: X \times \Aone \to X'$ is an $\Aone$-homotopy between them, then $f(\RR), g(\RR)$ are homotopic maps
 of varieties, via the homotopy obtained by restricting
 $H(\RR): X(\RR) \times \Aone(\RR) = X(\RR) \times \RR \to X'(\RR)$ to the subspace $X(\RR) \times [0,1]$.
\end{remark}

Using Lemma \ref{lem.weakequiv}, present $U(r; \AA^2)$ as the variety of $2r$-tuples \[(z_1, \dots, z_r,w_1, \dots,
w_r)\quad \text{ such
that } \quad (z_1, \dots, z_r) \neq (0, \dots, 0).\] This variety carries an action by $C_2$ sending $z_i \mapsto -z_i$ and
fixing the $w_i$. We know $U(r;\AA^2)$ and $B(r;\AA^2)$ are naively homotopy
equivalent to $\mathbb{A}^r\setminus \{0\}$ and $\mathbb{A}^r\setminus \{0\}/C_2$ respectively.

\begin{construction} \label{cons:YP} We now consider an inclusion that is not, in general, an equivalence. Let
 $P(r) = \Spec S$ denote the subvariety of $\A^r\sm \{0\}$ consisting of $r$-tuples $(z_1, \dots, z_r)$ such that
 $\sum_{i=1}^r z_i^2 = 1$. This is an $(r-1)$-dimensional closed affine subscheme of $\A^r \sm \{0\}$, invariant under
 the $C_2$ action on $\A^r \sm \{0\}$. The quotient of $P(r)$ by $C_2$ is $Y(r) = \Spec R$, and is equipped with an
 evident map $Y(r) \to (\AA^r \sm \{0\}) /C_2 \to B(r; \AA^2)$. Here $S$ and $R$ take on the same meanings as in
 Construction \ref{cons:Chase}.
\end{construction}

\begin{proposition} \label{prop:defRet} Let notation be as in Construction \ref{cons:YP}. The real manifold
 $B(r;\A^2)(\RR)$ has the homotopy type of
 \[ B(r; \A^2) (\RR) \simeq \RRP^{r-1} \amalg \RRP^{r-1} .\]
 The closed
inclusion $Y(r) \to B(r;\A^2)$ includes $Y(r)(\RR) \to B(r;\A^2)(\RR)$ as a deformation retract of one of the connected components.
\end{proposition}
\begin{proof} By Lemma \ref{lem.weakequiv} and Remark \ref{rem.finprod}, the manifold $B(r,\A^2)(\RR)$ is homotopy
 equivalent to $(\A^r\sm \{0\}/C_2)(\RR)$. The manifold $(\A^r\sm \{0\}/C_2)(\CC)$ consists of equivalence classes of $r$-tuples of
 complex numbers $(z_1,...,z_r)$, where the $z_i$ are not all $0$, under the relation
\[ (z_1,...,z_r) \sim (-z_1,...,-z_r).\]
The real points of $(\A^r\sm \{0\})/C_2$ consist of Galois-invariant equivalence classes. There are two components of this manifold: either
the terms in $(z_1,...,z_r)$ are all real or they are all imaginary. In either case, the connected component is
homeomorphic to the manifold $\RR P^{r-1}$.

 We now consider the manifold $Y(r)(\RR)$. This arises as the Galois-fixed points of $Y(r)(\CC)$, which in turn is the
 quotient of $P(r)(\CC)$ by a sign action. That is, $P(r)(\CC)$ is the complex manifold of $r$-tuples $(z_1, \dots, z_r)$
 satisfying $\sum_{i=1}^r z_i^2 = 1$. Again, in the $\RR$-points, the $z_i$ are either all real or all purely
 imaginary. The condition $\sum_{i=1}^r z_i^2 = 1$ is incompatible with purely imaginary $z_i$, so $Y(r)(\RR)$ is the
 manifold of $r$-tuples of real numbers $(z_1, \dots, z_r)$ satisfying $\sum_{i=1}^r z_i^2 =1$, taken up to sign. In
 short, $Y^r(\RR) = \RRP^{r-1}$.

 As for the inclusion $Y(r)(\RR) \to B(r;\AA^2)(\RR)$, it admits the following description, as can be seen by tracing through
 all the morphisms defined so far. Suppose given an equivalence class of real numbers $(z_1, \dots, z_r)$, satisfying
 $\sum_{i=1}^r z_i^2 = 1$, taken up to sign. Then embed $(z_1, \dots, z_r)$ as the point of $B(r;\AA^2)(\RR)$ given by the
 class of $(z_1, z_2, \dots, z_r, 0, \dots, 0)$. That is, embed $\RRP^{r-1}$ in
 $\RR^r \times \left(\RR^{r-1} \sm \{0\} \right) /C_2$ by embedding $\RRP^{r-1} \subset (\RR^r \sm \{0\})/ C_2$ as a
 deformation retract, and then embedding the latter space as the zero section of the trivial bundle. It is elementary
 that this composite is also a deformation retract.
\end{proof}

\begin{remark}
 We remark that the functor $X \leadsto X(\RR)$ does not commute with colimits. For instance $U(r;\AA^2)(\RR) / C_2$, which is
 connected, is not the same as $B(r;\AA^2)(\RR)$.
 
 In fact, the two components of $B(r;\AA^2)(\RR)$ as calculated above correspond to two isomorphism classes of quadratic
 \'etale $\RR$-algebras: one component corresonds to the split algebra $\RR \times \RR$, and the other to the nonsplit
 $\CC$.
\end{remark}

We will need two properties of $\Hoh^*(\RRP^r; \FF_2)$ here. Both are standard and may be found in \cite{hatcher}.
\begin{itemize}
\item $\Hoh^*(\RRP^r ;\FF_2) \iso \FF_2[\theta]/(\theta^{r+1})$ where $|\theta| = 1$.
\item The standard inclusion of $\RRP^r \hookrightarrow \RRP^{r+1}$ given by augmenting by $0$ induces the evident
 reduction map $\theta \mapsto \theta$ on cohomology.
\end{itemize}

\begin{proposition}
\label{prop.stablecoho} We continue to work over $k = \RR$. Let
$s_r:B(r;\AA^2)\to B(r+1;\AA^2)$ be the stabilization map of Definition \ref{def:stabilizeMaps}. The induced map on cohomology groups
\[ s_r^*: \Hoh^j(B(r+1;\AA^2)(\RR) ; \FF_2 ) \to \Hoh^j(B(r;\AA^2)(\RR); \FF_2) \]
is an isomorphism when $j \le r$ and is $0$ otherwise
\end{proposition}
\begin{proof}
 The map $s_r^*$ is arrived at by considering the inclusion $U(r;\AA^2) \to U(r+1;\AA^2)$, which is given by augmenting an
 $r$-tuple of pairs $(a_1, b_1, \dots, a_r, b_r)$ by $(0,0)$, and then taking the quotient by $C_2$. After
 $\RR$-realization, one is left with a map $B(r;\AA^2)(\RR) \to B(r+1;\AA^2)(\RR)$ which on each connected component is homotopy
 equivalent to the standard inclusion $\RRP^r \to \RRP^{r+1}$. The result follows. 
\end{proof}

\begin{proposition}[Ojanguren]
 Let $S$ and $R$ be as in Construction \ref{cons:Chase}. The quadratic \'etale algebra $S/R$ cannot be generated by
 fewer than $r$ elements.
\end{proposition}

\begin{proof}[Sketch of proof]
 Write $Y(r) = \Spec R$ as in Construction \ref{cons:YP}. The morphism $Y(r) \to B(r ; \AA^2)$ of Construction \ref{cons:YP} classifies a quadratic
 \'etale algebra over $Y(r)$, and we can identify this algebra as $S$.

 The map $\phi: Y(r) \to B(r; \AA^2)$ induces stable maps $\tilde \phi: Y(r) \to B(r; \AA^2)$. Any such stable map
 induces a surjective map
 \[ \tilde \phi^*: \Hoh^*( B(r; \AA^2)(\RR); \FF_2) \to \Hoh^*(Y(r)(\RR); \FF_2) \]
 by Proposition \ref{prop:defRet} and \ref{prop.stablecoho}. In particular, it is a surjection when $\ast = r-1$.

Suppose $S$ can be generated by $r-1$ elements, then there
 is a classifying map $\phi' : Y(r) \to B(r-1; \AA^2)$, from which one can produce a stable map 
 \[(\tilde \phi')^* : \Hoh^*( B(r ; \AA^2)(\RR): \FF_2) \to \Hoh^*(B(r-1; \AA^2); \FF_2) \to \Hoh^*(Y(r)(\RR) ;
 \FF_2). \]
 By reference to Corollary \ref{cor.htylem}, for sufficiently large values of $R$, the maps $\tilde \phi^*$ and
 $(\tilde \phi')^*$ agree. But $(\tilde \phi')^*$ induces the $0$-map when $\ast = r-1$, since $\Hoh^*(B(r-1;
 \AA^2)(\RR); \FF_2)$ is a direct sum of two copies of $\FF_2[\theta]/(\theta^{r-1})$. This contradicts the
 surjectivity of $\tilde \phi^*$ in this degree.
\end{proof}

\subsection{Algebras over fields containing a square root of \texorpdfstring{$-1$}{-1}}

 \begin{remark} \label{rem:710}
 When the field $k$ contains a square root $i$ of $-1$, the analogous construction to that of Chase exhibits markedly different
 behaviour. For simplicity, suppose $r$ is an even integer. Consider the ring
 \[ S'=\dfrac{k[z_1, \dots , z_r]}{\big(\sum_{i=1}^r z_i^2-1\big)} \]
 with the action of $C_2$ given by $z_i \mapsto -z_i$. Let $R' = (S')^{C_2}$. After making the change of variables
 $x_{j} =z_{2j-1}+iz_{2j}$ and $y_j=z_{2j-1} - iz_{2j}$, we see that $S'$ is isomorphic to
 \[\frac{k[x_1, \dots, x_{r/2}, y_1, \dots, y_{r/2}]}{\Big( \sum_{j=1}^{r/2} x_jy_j -1 \Big)}\]
 and $R'$ is isomorphic to the subring consisting of terms of even degree. The smallest $R'$-subalgebra of $S'$
 containing the $r/2$-terms $x_1, \dots, x_{r/2}$ contains each of the $y_j$ because of the relation
 \[ y_j = \sum_{l=1}^{r/2} x_l (y_l y_j) \]
 so $S'$ may be generated over $R'$ by $r/2$ elements. In fact, $R'$ is the coordinate ring of $DQ_{r-1}$, by Lemma
 \ref{lem:AffDQ}. In Proposition \ref{rem.affapp} below, we show that $S'$ cannot be
 generated by fewer than $r/2$ elements over $R'$.
\end{remark}

One may reasonably ask therefore, over a field $k$ containing a square root of $-1$:
\begin{question}
For a given dimension $d$, is there a smooth $d$-dimensional affine variety $\Spec R$ and a finite \'etale algebra $\sh A$ over
 $\Spec R$ such that $\sh A$
 cannot be generated by fewer than $d +1$ elements?
\end{question}

The result of \cite{first} implies that if $d+1$ is increased, then the answer is negative.

\begin{remark}
 If $d=1$, the answer to the question is positive. An example can be produced using any smooth affine curve $Y$ for
 which ${}_2\Pic(Y) \neq 0$. Specifically, one may take a smooth elliptic curve and discard a point to produce such a
 $Y$. A nontrivial $2$-torsion line bundle $\sh L$ on $Y$ cannot be generated by $1$ section, since it is not
 trivial. One may choose a trivialization $\phi: \sh L \tensor \sh L \to \sh O$, and therefore endow
 $\sh L \oplus \sh O$ with the structure of a quadratic \'etale algebra, as in Construction \ref{cons:AfromL}, and this algebra
 also cannot be generated by $1$ element.
\end{remark}

\begin{proposition} \label{rem.affapp}
 Let $k$ be a field containing a square root $i$ of $-1$. Let $T$ denote the ring 
 \[ T = \frac{ k [x_1 ,\dots, x_r, y_1, \dots, y_r] }{ ( \sum_{i=1}^r x_i y_i - 1 ) } \]
 endowed with the $C_2$ action given by $x_i \mapsto -x_i$ and $y_i \mapsto -y_i$. Let $R =T^{C_2}$. Then
 the quadratic \'etale algebra $T$ over $R$ can be generated by the $r$ elements $x_1, \dots, x_r$, but cannot
 be generated by fewer than $r$ elements.
\end{proposition}
\begin{proof}
 The ring $R$ is the coordinate ring of the variety $DQ_{2r-1}$ in Lemma \ref{lem:AffDQ}. In particular, there is an
 $\Aone$-equivalence $\phi: DQ_{2r-1} \to B(r; \AA^2)$, as in equation \eqref{eq:1}. Tracing through this composite, one sees
 it classifies the quadratic \'etale algebra generated by $x_1, \dots, x_r$, i.e., $T$ itself---the argument being as
 given for $DQ_{r-1}$ in Remark \ref{rem:710}.

 Suppose for the sake of contradiction that $T$ can be generated by $r-1$ elements over $R$. Let
 $\phi' : DQ_{2r-1} \to B(r-1; \AA^2)$ be a classifying map for some such $r-1$-tuple of generators. Let $\tilde \phi$
 and $\tilde \phi'$ denote the composite maps $DQ_{2r-1} \to B(2r-1; \AA^2)$. By Corollary \ref{cor.htylem}, these maps
 induce the same map on Chow groups. But in degree $r-1$, the map
 $\tilde \phi^* : \CH^{r-1} ( B(2r-1; \AA^2)) \to \CH^{r-1}( B(r; \AA^2) ) \to \CH^{r-1}(DQ_{2r-1})$ is an isomorphism of
 cyclic groups of order $2$, by reference to Proposition \ref{pr:DItwo}, while by the same proposition,
 $(\tilde \phi')^* : \CH^{r-1} ( B(2r-1; \AA^2)) \to \CH^{r-1}( B(r-1; \AA^2) ) \to \CH^{r-1}(DQ_{2r-1})$ is $0$.
\end{proof}

The following shows that the bound of \cite{first} is not quite sharp when applied to quadratic
\'etale algebras over smooth $\bar k$-algebras where $\bar k$ is an algebraically closed field.

\begin{proposition}
 Let $\bar k$ be an algebraically closed field. Let $n\geq 2$, and $\Spec R$ an $n$-dimensional smooth affine
 $\bar k$-variety. If $\sh A$ is a quadratic \'etale algebra on $\Spec R$, then $\sh A$ may be generated by $n$ global
 sections.
\end{proposition}
\begin{proof}
 Let $\sh L$ be a torsion line bundle on $\Spec R$, or, equivalently, a rank-$1$ projective module on $R$. A result of
 Murthy's, \cite[Corollary 3.16]{murthy}, implies that $\sh L$ may be generated by $n$ elements if and only if
 $c_1(\sh L)^n = 0$. By another result of Murthy's, \cite[Theorem 2.14]{murthy}, the group $\CH^n(R)$ is torsion free,
 so it follows that if $\sh L$ is a $2$-torsion line bundle, then $\sh L$ can be generated by $n$ elements. The
 proposition follows by Proposition \ref{pr:generationEquivalence}.
\end{proof}

\section*{Acknowledgements} This paper owes several great debts to Zinovy Reichstein, who introduced each author,
separately, to the question at hand and who also supplied the argument in the introduction reducing the question of
\'etale algebras of degree $n$ to that of degree $2$. We would also like to thank Uriya First, who very graciously read
an earlier draft. We would like to thank Manuel Ojanguren who read an early draft of this paper, explained the
construction of the example he had given to Uriya First, and encouraged the authors.

\bibliographystyle{plainurl}
\bibliography{etalebibliography}

\end{document}